\documentclass[a4paper, 10pt]{amsart}
\usepackage{amsmath, latexsym,  amssymb, amscd}
\usepackage[all]{xy}
\newcommand{\ecuno}{\epsilon_1(V,\eta)}
\newcommand{\ecdue}{\epsilon_2(V,\eta)}

\newcommand{\stab}{{\rm{Stab}}\,\,}

\newcommand{\stabv}{{\rm{Stab}}\,\,V}
\newcommand{\stabvz}{{\rm{Stab}^0}V}

\newcommand{\cgf}{c_1}
\newcommand{\cgw}{c_2}

\newcommand{\pippo}{\psi}
\newcommand{\fiffo}{\phi}

\newcommand{\expm}{{d+1}}
\newcommand{\expalt}{{1+\frac{1}{2n}}}
\newcommand{\expaltu}{{1+\frac{1}{2n}}}

\newcommand{\valeta}{{2d}}
\newcommand{\coeta}{{2(g-d-1)}}
\newcommand{\coefa}{{{d}+2d\eta}}

\newcommand{\emor}{\rm End }

\newcommand{\indice}{r}

\newcommand{\g}{{\theta}}

\newcommand{\cod}{\rm cod }

\newcommand{\qe}{\mathbb{Q}}

\newcommand{\ze}{\mathbb{Z}}

\newcommand{\upxi}{\zeta}

\newcommand{\varieta}{V}

\newcommand{\punto}{y}

\newcommand{\euno}{\varepsilon_3}

\newcommand{\ecinque}{\varepsilon_1}

\newcommand{\kzero}{{K}}

\newcommand{\kdue}{K+||p||}

\newtheorem{thm}{Theorem}[section]
\newtheorem{con}[thm]{Conjecture}
\newtheorem{propo}[thm]{Proposition}
\newtheorem{lem}[thm]{Lemma}

\newtheorem{D}[thm]{Definition}
\newtheorem{remark}[thm]{Remark}

\newtheorem*{propa*}{Proposition B}
\newtheorem*{propb*}{Proposition A}

\title[{ Non-dense subsets of   varieties } ]
{Non-dense subsets of   varieties \\ in a power of an elliptic curve}

\author[{ Evelina Viada}]{  
 }

\begin{document}

\maketitle
\centerline{Evelina Viada\footnote{Evelina Viada,
     Universit\'e de Fribourg Suisse, P\'erolles, D\'epartement de Math\'ematiques, Chemin du Mus\'ee 23, CH-1700 Fribourg,
Switzerland, viada@math.ethz.ch,
    evelina.viada@unifr.ch.}
    \footnote{Supported by the SNF (Swiss National Science Foundation).}
}

Let $E$ be an elliptic curve without  C.M. defined over $\overline\qe$.
We  show that on a transverse $d$-dimensional variety $V\subset E^g$, the set of algebraic points of bounded height  which are close to the union of all algebraic subgroups of $E^g$ of codimension $d+1$ translated by a point in a subgroup $\Gamma$ of $E^g$ of finite rank, is non-Zariski dense in $V$. The notion of close is defined using a height function. If $\Gamma=0$, it is sufficient to assume that $V$ is weak-transverse.
This result is optimal with respect to   the codimension of the algebraic subgroups.

The method is based on an essentially optimal effective version of the Bogomolov Conjecture. Such an effective result is proven for subvarieties of $E^g$. If we assume that the sets have bounded height, then we can prove that they are not Zariski dense. 
A conjecture, known in some special cases, claims that the sets in question  have bounded height. We prove here a new case. 
In conclusion, our results prove a generalized case  of a conjecture by Zilber and by Pink in $E^g$.

\section{introduction}

In this article all algebraic varieties are defined over $\overline{\qe}$ and we consider only algebraic points.
Denote by  $A$  an abelian variety     of
dimension $g$. Consider a proper irreducible algebraic  subvariety   $V$
of  $A$ of dimension $d$. 
 We say that:
\begin{itemize}
 \item $V$ is transverse, if $V$ is not contained in any
translate of a proper algebraic subgroup of $A$.

\item  $V$ is
weak-transverse, if $V$ is not contained in  any proper algebraic
subgroup of $A$.
\end{itemize}

Given an integer $ \indice$ with $1 \le \indice \le g$ and  a subset $F$ of $A$, we define the set
$$S_{\indice}(V,F)= V \cap  \bigcup_{\mathrm{cod}B \ge \indice} (B+F),$$
where $B$ varies over all abelian subvarieties of $A$ of
codimension at least $\indice$ and $$B+F=\{b+f \,\,\,: \,\,\,b\in B, \,\,\,f\in F\}.$$ 
Note that $$S_{r+1}(V,F)\subset S_{r}(V,F).$$
We  denote the set $S_{\indice}(V, A_{\rm Tor})$
simply  by $S_{\indice}(V)$, where $A_{\rm Tor}$ is the torsion of $A$.
For convenience, for $r>g$ we  define $S_r(V,F)=\emptyset$ and for $V^e$ a subset of $V$  we define
$$S_r(V^e,F)=V^e\cap S_r(V,F).$$
We ask   for which sets  $F$ and integers
$\indice$ the set 
 $S_{\indice}(V,F)$ has bounded height or  is non-Zariski    dense in $V$.
 
Depending on the choice of $F$, the set $S_g(V,F)$ appears in the literature in the
context of the Mordell-Lang, of the Manin-Mumford and of the Bogomolov Conjectures.
More recently  Bombieri, Masser and Zannier \cite{BMZ}   proved that
for a transverse
 curve in a torus, the set $S_2(C)$ is finite. They
investigated  for the first time   intersections with the union of all
algebraic subgroups of a given codimension. This opens a vast
number of conjectures for subvarieties of semi-abelian varieties.

In this paper  we consider a variety in a power of an elliptic curve. In the first part of this work we study the non-density of $S_{d+1}(V,\cdot)$, the last part  is dedicated to  its  height.
Let $E$ be  an elliptic curve without C.M. Consider on $E^g$ the line bundle $\mathcal{L}$ given as tensor product of the pull backs via the natural projections of a symmetric ample line bundle on  $E$.
We fix on  $E^g $ a semi-norm  $||\cdot||$ induced by the N\'eron-Tate  height on $E$.
For $\varepsilon \ge 0$, we denote $$\mathcal{O}_{\varepsilon}=\{ \xi \in E^g  : ||\xi|| \le \varepsilon\}.$$ 
We denote by $\Gamma$  a subgroup of finite rank in  $E^g $.  We define
$\Gamma_{\varepsilon}=\Gamma + \mathcal{O}_\varepsilon.$
 
 Let $V$ an irreducible algebraic subvariety of $E^g$ of dimension $d$. For a non negative real $\kzero$, we define $$V_\kzero= V \cap \mathcal{O}_\kzero.$$

Our main result is:

\begin{thm}
\label{main}

 For every $\kzero\ge0$ there  exists  an effective $\varepsilon>0$ such that: 

\begin{enumerate}

\item  If  $V$ is weak-transverse,  $S_{d+1}(V_\kzero,\mathcal{O}_\varepsilon)$  is non-Zariski    dense in $V$.

\item  If  $V$ is transverse, $S_{d+1}(V_\kzero,\Gamma_\varepsilon)$ is non-Zariski    dense in $V$.
\end{enumerate}

\end{thm}
Because of the different hypotheses on the variety and
the different sets in the thesis, there are no evident implications between the statements i. and ii.

Let us say at once that the theorem is expected to hold for $V$ instead of $V_K$. This is immediately implied by:
 \begin{con}
\label{hb}
There exist  $\varepsilon>0$
  and a non-empty Zariski open subset $V^u$ of $V$ such that: 
\begin{enumerate}
\item  If $V$ is weak-transverse, 
  $S_{d+1}(V^u,\mathcal{O}_\varepsilon)$  has bounded height.

\item  If $V$ is transverse,   $S_{d+1}(V^u,\Gamma_\varepsilon)$ has bounded height.
\end{enumerate}
\end{con}

The  method known to show that the height is bounded relies on a Vojta inequality, unless $\Gamma$ is trivial.  This method gives
optimal results for curves, while for varieties a hypothesis stronger than transversality are needed. Let $V\subset E^g$ be a variety of dimension $d$ such that \begin{equation}
\label{stella}
\dim (V+B)=\min(d+\dim B, \,\,\, g )
\end{equation} for all abelian subvarieties $B$ of $E^g$.  
In this paper we   extend the    proof of R\'emond  of  Conjecture \ref{hb} ii. for $V$ satisfying condition  (\ref{stella}). In Theorem  \ref{wtalt}, we prove Conjecture \ref{hb} i. for $V\times p$ where  $p\in E^s$ is a point  not lying in any proper algebraic
subgroup of $E^s$.   We can then conclude:
\begin{thm}
\label{conseguenza}

For $V$ satisfying condition (\ref{stella}) and  $p\in E^s$ a point  not lying in any proper algebraic
subgroup of $E^s$, there exists $\varepsilon>0$ such that:
\begin{enumerate}
\item The set $S_{d+1}(V\times p,\mathcal{O}_\varepsilon)$ is non-Zariski dense in $V\times p$.
\item The set $S_{d+1}(V,\Gamma_\varepsilon)$ is non-Zariski dense in $V$.
\end{enumerate}

\end{thm}
In section 2,  we clarify that, up to an isogeny of $E^n$, a weak-transverse variety in $E^n$ has the shape $V\times p$ for $V$ transverse in some $E^g$ and $p$  a point in $E^{n-g}$ not lying in any proper algebraic subgroup of $E^{n-g}$.

 For the codimension of the subgroups equal to $g$, statements i. and ii.  are cases of   the  Bogomolov Conjecture   and   the Mordell-Lang plus Bogomolov Conjecture  respectively.
Let us 
emphasise that our theorem neither gives a new proof of the
Bogomolov Conjecture (as we make use of such a result), nor we get a new proof of the Mordell-Lang Conjecture (as we use a more general Vojta inequality).
On the contrary we give a new proof of the Mordell-Lang plus Bogomolov Theorem (Poonen \cite{poonen}), under the
assumption (\ref{stella}). In addition, theorem \ref{conseguenza}  part ii. proves a case of  a conjecture by Zilber and Pink extended by the   Bogomolov Conjecture.\\

 In \cite{io}, we proved our main result for a
 curve in  $E^g$. A naive extension of  the method in \cite{io},  would imply a weak form of Theorem \ref{main}, where the codimension of the algebraic subgroups shall be at least $2d$ instead of $d+1$.
Here, we improve the method used in \cite{io} obtaining the optimal $d+1$. 
In the first instance we show that Theorem \ref{main} i. and ii. are equivalent,
then we prove  Theorem \ref{main} ii. 
\begin{thm}
\label{equii}
Given $\kzero\ge0$ and a positive integer $r$, the following statements are equivalent:
\begin{enumerate}

\item  For $V$  weak-transverse,  there exists  $\varepsilon>0$ such that  $S_{r}(V_\kzero,\mathcal{O}_\varepsilon)$  is non-Zariski    dense in $V$.

\item  For $V$  transverse,  there exists  $\varepsilon>0$ such that  $S_{r}(V_\kzero,\Gamma_\varepsilon)$ is non-Zariski    dense in $V$.

\end{enumerate}

\end{thm}

We shall then prove Theorem \ref{main} part ii. Like for curves, 
the strategy of the  proof is based on two steps.
A union of infinitely many sets is non-Zariski dense if:
\begin{itemize}
 \item[(1)] the  union can be taken over    finitely
many sets,
\item[(2)] all sets in the union are
non-Zariski dense.
\end{itemize}

Part (1)  is a 
typical problem of Diophantine approximation; we approximate an
algebraic subgroup with a subgroup of bounded degree (see Proposition \ref{centro}).
 
The second step (2) is a problem of   height theory and its proof relies on  an essentially optimal lower bound for the normalized height of a transverse subvariety in $E^g$,  Theorem
\ref{sin2} below. This part is delicate. The dimension of the
variety intervenes  heavily on the estimates we provide. A fundamental idea is to reduce the problem to the study of varieties with finite stabilizer (see section \ref{due}).\\

We define 
$\mu(\varieta)$ as the supremum
of the   reals $\epsilon(\varieta)$ such that $S_g(\varieta,\mathcal{O}_{\epsilon(\varieta)})=\varieta\cap
\mathcal{O}_{\epsilon(\varieta)}$ is non-Zariski dense in $V$. Work by  Ullmo \cite{ulm} and Zhang \cite{zang} proves the  Bogomolov Conjecture. This shows that   $\mu(V)>0$, for $V$ transverse. A first effective lower bound
for  $\mu(\varieta)$  is
provided by S. David and P. Philippon  \cite{sipa}  Theorem 1.2.
The type of bounds we need are an  elliptic analogue of  Amoroso and David \cite{fra} Theorem 1.4. Such a result is proven by Galateau in his Ph.D. thesis for $d\ge g-2$, and  in a preprint \cite{aur} for varieties in a product of elliptic curves with or without C.M. (he gives estimate for the square of  $\mu(V)$).

\begin{thm}
[Bogomolov type bound, Galateau \cite{aur}] \label{sin2}
Let $V$ be a transverse subvariety  of $E^g$ of codimension $\cod V$ defined over $\overline{\qe}$. 
For  $\eta>0$, there exists a positive effective constant $ c(E^g,\eta)$ depending on the ambient variety  and $\eta$,  such that 
for  $$\epsilon(V,\eta) = \frac{c(E^g,\eta)} {(
\deg_{{\mathcal{L}}} V)^{\frac{1}{2\cod V}+\eta}}$$
the set 
$$V(\overline{\mathbb{Q}}) \cap \mathcal{O}_{\epsilon(V,\eta)}$$ is non-Zariski    dense in $V$.

\end{thm}
The bound  $\epsilon(V,\eta)$ depends on the invariants of the ambient variety  and  on the degree of $V$.
The quasi optimal dependence  on the degree of $V$ and the non-dependece on the  field of definition and height of   $V$ are of
crucial importance for our application.

The non-Zariski density for transverse varieties  has often been
investigated with  the method introduced by Bombieri, Masser and Zannier in
\cite{BMZ}. To show
  the non-density property 
they  use an essentially optimal   Generalized Lehmer Conjecture.
 In \cite{V} we applied their method  to a transverse curve, $\Gamma=0$ and $\varepsilon=0$. In \cite{RV} R\'emond and the author  extended the method to transverse curves, $\varepsilon=0$ and any  $\Gamma$ of finite rank. In \cite{Gfin}-\cite{Gprep}
  R\'emond generalized it to varieties satisfying   a  geometric property  stronger than
 transversality.

The main advantage of using a Bogomolov instead of a Lehmer type bound is   that an essentially optimal Generalized Lehmer Conjecture is proven for C.M. abelian varieties and it is not likely to be proven in a near
future for non C.M. abelian varieties. On the contrary  the
Bogomolov type bound is proven
at least for some non C.M. abelian varieties. 
In addition, our method gives  the non-density for a neighbourhood  of
positive radius $\varepsilon$. At present it is not known how to obtain results of this kind in abelian varieties using a Lehmer type bound.

The non-Zariski density for a transverse subvariety in a torus and $\Gamma=0$ has
been  studied  independently by P.  Habegger  \cite{PHIL}. He uses the
Bogomolov type bound proven by Amoroso and David \cite{fra} and 
proves that for
a transverse variety  $V$  in $\mathbb{G}^n_m$,  there
exists $\varepsilon>0$ such that the set
$S_{2d}(V,\mathcal{O}_\varepsilon)$ is non-Zariski dense. \\

In the next section we fix the notation and  recall the results we need from \cite{io}. In section 3 we present the four main steps of the proof of Theorem \ref{main}. Section \ref{due} is the core of this article: we prove the non-density of the intersections. In section 5 we conclude the proof of the main theorem. In the final section we prove that sometimes the height is bounded.\\

\noindent{\bf Acknowledgements.}
I kindly  thank the Referee for his valuable suggestions.

\newpage

 \section{preliminaries}

In the following, we aim to be as transparent as possible, polishing statements from technicality. Therefore,
 we  present the proofs for a power of an elliptic curve $E$  without C.M. Then $\emor(E)$ is identified with $\ze$. Proofs for a subvariety in a product of arbitrary elliptic curves are slightly more technical.

\subsection{Small points}
On $E$, we fix a symmetric very ample line bundle $\mathcal{L}_0$. On
$E^g$, we consider the bundle $\mathcal{L}$ which is  the tensor product of the
pull-backs of $\mathcal{L}_0$ via the natural projections on the
factors. Degrees are computed with respect to the polarization $\mathcal{L}$. 
Usually $E^g$ is endowed with the $\mathcal{L}$-canonical N\'eron-Tate height
$h'$. Though,  we prefer to 
define  on $E^g$ the height of the maximum 
\begin{equation*}
h(x_1,\dots ,x_g)=\max_i(h(x_i)),
\end{equation*}
where $h(x_i)$  on $E $  is given by the  $\mathcal{L}_0$-canonical N\'eron-Tate height.
Note that  $h(x)\le h'(x)\le gh(x)$. Hence, the two norms induced by $h$ and $h'$ are equivalent.
We denote by $||\cdot||$ the semi-norm induced by $h$ on $E^g $.

For $\varepsilon \ge 0$, we denote 
$$\mathcal{O}_{\varepsilon}=\{ \xi \in E^g  : ||\xi|| \le \varepsilon\}.$$

\subsection{Morphisms and their height}

We denote by $M_{r,g}(\ze)$ the  module of $r\times g$ matrices with
entries in  $\ze$.
For $F=(f_{ij})\in M_{r,g}(\ze)$,  we define the height of $F$ as the maximum of the absolute value of its entries
$$H(F)=\max_{ij}|f_{ij}|.$$
A morphism $\phi:E^g \to E^r$ is identified with an integral matrix. 
Let $a\in \ze $, we denote by $[a]$ the multiplication by $a$.

 Note that, the set of morphisms of height less than a  constant is a finite set.
 
\subsection{Algebraic subgroups}

Let $B$ be an algebraic subgroup of $E^g$ of codimension $r$. Then $B \subset \ker \phi_B$ for a  surjective morphism $\phi_B:E^g\to E^r$.
 Conversely, we denote by $B_\phi$ the kernel of a surjective morphism $\phi:E^g \to E^r$. Then $B_\phi$  is an algebraic subgroup of $E^g$ of codimension $r$.  Note that $r$ is the rank of $\phi$.
 An easy observation (see for instance \cite{V} page  61 line -3) gives that each  of the $r$ equations defining $B_\phi$ has degree at most $H(\phi)^2$, up to a multiplicative constant  depending  on $\deg E$ and $g$. This directly implies:
  \begin{lem}
 \label{degb}
 Let $\phi:E^g \to E^r$ be a surjective morphism. Then
$$\deg B_\phi\le c_0 H(\phi)^{2r}$$
 where $c_0$ is a constant depending  on $\deg E$ and $g$.
 \end{lem}

 
 
\subsection{Subgroups}

Let $\Gamma$ be a subgroup  of $E^g $ of finite rank $s$. Then $\Gamma$ is a $\ze$-module of rank $s$. We call a maximal free set  of $\Gamma$ a set of $s$ linearly independent elements of $\Gamma$, in other words a basis of $\Gamma \otimes_\ze \qe$.  If $\Gamma$ is a free module, we call integral generators a set of $s$ generators of $\Gamma$.

The  division group ${\Gamma}_0$ of the coordinates group of the points of $\Gamma$, in short of $\Gamma$, is a subgroup of $E $ defined as 
\begin{equation}
\label{gammazero}
{\Gamma}_0=\{y  \in E   {\rm {\,\,\,such\,\,\,that\,\,\,}}Ny\in\pi( \Gamma){\rm \,\,\,for\,\,\,}  N\in \ze^*{\rm \,\,\,and\,\,\,}\pi:E^g \to E\}.
\end{equation}

Note that, $\Gamma^{g}_0=\Gamma_0\times \dots \times \Gamma_0$  contains $\Gamma$ and it is a module of finite rank. This shows that, to prove non-density statements for $\Gamma$ it is enough to prove them for $\Gamma_0^{g}$.

\begin{D}
\label{rankp}
We say that a point  $p=(p_1,\dots,p_n)\in E^n $ has rank $s$ if  its coordinates group $\langle p_1,\dots ,p_n\rangle$ has rank $s$.   We define $\Gamma_p$ to be the division group of   $\langle p_1,\dots ,p_n\rangle$.
\end{D}

Given a point $p\in E^s$ of rank $s$,  we associate to $p$ a positive real $\varepsilon_0(p)$.  
This value will be used several times in the following.

\begin{propo}[\cite{io} Proposition 3.3 with $\tau=1$, $\emor(E)=\ze$,
   $c_0(p)=c_2(p,1)$ and $\varepsilon_0(p)=\varepsilon_0(p,1)$]
Let  $p_1, \dots, p_s$ be  linearly independent points of
$E $ and $p=(p_1, \dots, p_s)$.  
Then, there exist  positive reals $c_0(p)$   and  $ \varepsilon_0(p)$ such that 
$$c_0(p)\sum_i|b_i|^2||p_i||^2 \le \Big|\Big|\sum_ib_i(p_i-\xi_i)-b\upxi \Big|\Big|^2$$
 for all  $b_1,\dots,b_s,b\in \ze$ with $|b|\le \max_i|b_i|$ and  for all
$\xi_1,\dots, \xi_s ,\upxi \in E $ with $||\xi_i||, ||\upxi|| \le
 \varepsilon_0(p)$.

\end{propo}

\subsection{From transverse to weak-transverse}
Let $V$ be transverse in $E^g$ and let $\Gamma$ be a subgroup  of $E^g $ of finite rank. Let $\Gamma_0$ be the division group of $\Gamma$ and let $s$ be its rank.  If $s=0$ we define $V'=V$. If $s>0$, we denote by $\gamma_1, \dots , \gamma_s$ a maximal free set  of $\Gamma_0 $ and  $$\gamma=(\gamma_1, \dots , \gamma_s).$$
 We define
 $$V'=V\times \gamma.$$
Since $V$ is transverse and $\gamma$ has rank $s$, then $V'$  is weak-transverse  in $E^{g+s}$.

\subsection{ From  weak-transverse to transverse}
Let $V'$ be weak-transverse in $E^{n}$.  If $V'$ is transverse then we define $V=V'$ and $\Gamma=0$. If $V'$ is not transverse, let $H_0$ be the abelian
subvariety of smallest dimension $g$ such that $V'\subset H_0+p^\perp$ for
$p^\perp\in H_0^\perp$ and
 $H_0^\perp$  the orthogonal complement of $H_0$ of dimension $s=n-g$.
Then  $E^n$ is isogenous to $ H_0\times H_0^{\perp}$. Furthermore $H_0$
 is isogenous to $E^g$ and $H_0^{\perp}$ is isogenous to $E^s$.  Let
 $j_0$, $j_1$ and $j_2$ be such  isogenies. 
 We fix the isogeny $$j=(j_1\times j_2)\circ j_0:E^n \to H_0\times H_0^{\perp}\to E^g\times E^s,$$
which sends $H_0$ to $E^g\times 0$ and $H_0^{\perp}$ to $0\times E^{s}$ and $j(p^\perp)=(0,
\dots ,0, p_1,\dots,p_s)$.
Since $V'$ is weak-transverse and defined over $\overline{\qe}$,
  $p=(p_1,\dots,p_s)$ has rank $s$ and is defined over $\overline{\qe}$.

We consider the natural projection on the first $g$ coordinates
\begin{equation*}
\begin{split}
 \pi:&E^g\times E^s \to E^g\\
 &j(V') \to \pi(j(V')).
 \end{split}
 \end{equation*}
 We define
$$V=\pi(j(V'))$$ and
$$\Gamma=\Gamma_p^g.$$

 Since $H_0$ has minimal dimension, the variety $V$ is transverse in $E^g$ and $\Gamma$ has  rank $gs$.
 Finally $$j(V')=V\times p.$$
We remark that we have  defined a bijection  $(V,\Gamma_0^g) \to V'$, which is exactly  what interest us.
\subsection{ Weak-transverse up to an isogeny}
\label{trwtr}
Statements on boundedness of heights and non-density  of sets  are  invariant under an
isogeny of the ambient variety. Namely, given an isogeny $j$ of $E^g$,  Theorem \ref{main} and Conjecture \ref{hb} hold for a variety  if and only if they hold for its image via $j$. Thus,
the previous discussion shows that without loss of generality, we can assume that a weak-transverse variety $V'$ in $E^{n}$ is of the form
$$V'=V\times p$$
where
\begin{enumerate}
\item $V$ is transverse in $E^g$, 
\item    $p=(p_1, \dots ,p_s)$  is  a point in $E^s$ of rank $s$,
\item $n=g+s$.

\end{enumerate}

In short we will say that $V\times p$ is a weak-transverse variety in $E^{g+s}$, to say that $V$ is transverse in $E^g$ and $p\in E^s$ has rank $s$.
This simplifies the setting for weak-transverse varieties.

\subsection{Gauss-reduced morphisms}
 The matrices  in $M_{r\times g}(\ze)$ of the form 
 \begin{equation*}
\phi=(aI_r|L)=
 \left(\begin{array}{cccccc}
a&\dots &0&a_{1,r+1}&\dots &a_{1,g}\\
\vdots & & \vdots & \vdots& & \vdots\\
0&\dots &a&a_{r,r+1}&\dots &a_{r,g}
\end{array}\right),
\end{equation*} 
with $ H(\phi)=a$ and no common factors of  the entries will play a key role in this work.  If $r=g$ simply forget $L$. The following definition of Gauss-reduced is slightly more general than the one given in \cite{io}, namely we omit here the assumption that the entries of the matrix have no common factors. This is a marginal simplification,  overseen in that paper.
\begin{D}[Gauss-reduced Morphisms]
\label{gr} 
Given positive integers $g, r $, we say that a  morphism $\phi:E^g \to E^r$ is Gauss-reduced  if: 
\begin{enumerate}
\item  There esists $a \in \mathbb{N}^*$ such that $aI_{r}$ is a submatrix of $\phi$, with $I_r$ the r-identity matrix,
\item $H(\phi)=a$. 

\end{enumerate}


\end{D}

A morphisms $\phi'$, given by a reordering of the rows of a
morphism $\phi$, has the same kernel as $\phi$. Saying that $aI_r$ is a sub-matrix  of $\phi$ fixes one permutation of the rows of $\phi$.

A reordering of the columns corresponds, instead, to a permutation of the coordinates.
  Statements  will be proven for Gauss-reduced morphisms of the form $\phi=(aI_r|L)$. For each other reordering of the columns  the proofs are analogous. 
Since there are  finitely many permutations  of $g$ columns, the non-density statements will follow. 

There are few easy facts that one shall keep in mind. Let $\psi:E^g \to E^r$ be a morphism and $\phi:E^g \to E^r$ be a Gauss-reduced morphism, then
\begin{enumerate}
\item For $x \in E^g $,  $$||\psi(x)||\le
  gH(\psi) ||x||$$ and  $$||\phi(x)||\le
  (g-r+1)a||x||.$$
\item For $x \in E^r\times \{0\}^{g-r}$,   $$\phi(x)=[a]x.$$
\end{enumerate}

The following lemma shows that every abelian subvariety of codimension
$r$ is contained in the kernel of a Gauss-reduced morphism of rank $r$.

\begin{lem} [\cite{io} Lemma 4.4 ii. with $\emor(E)=\ze$]
\label{tor}

Let $\psi:E^g\to E^r$ be a morphism of rank $r$. Then,
there exists a Gauss-reduced  morphism $\phi:E^g \to E^r$  such that 
$$B_{\psi}\subset B_{\phi}+(E^r_{\rm{Tor}}\times \{0\}^{g-r}).$$

\end{lem}

Taking intersections with    $V_{K}$,
the previous lemma translates  immediately as:
\begin{lem}
\label{tor1}

For any reals $K,\varepsilon\ge1$ and integer $r\ge1$, it holds 
$$S_r(V_{K},(\Gamma_0^g)_\varepsilon)=\bigcup_{\substack{\phi:E^g \to E^r\\ {\rm{Gauss-reduced}}} }V_{K} \cap (B_\phi+(\Gamma_0^g)_{\varepsilon} ).$$

\end{lem}

\subsection{  Quasi-special and Special Morphisms}

Special morphisms play a key role in the study of  weak-transverse varieties. A Special morphism $\tilde\phi$ is Gauss-reduced. In addition  the multiplication by $H(\tilde\phi)$ acts on some of the first $g$-coordinates.

\begin{D}[Quasi-special and Special Morphisms]
\label{spec} Given positive  integers $g,s,r$,
a morphism  $\tilde{\phi}:E^{g+s} \to E^r$ is Quasi-special  if
 there exist   a Gauss-reduced morphisms  $\phi:E^g\to E^r$ and a morphism $\phi':E^s \to E^r$ such that 
\begin{itemize}
\item[i.] $\tilde{\phi}=(\phi|\phi')$.
\end{itemize}
The morphism $\tilde{\phi}:E^{g+s} \to E^r$ is Special if it satisfies the further condition
\begin{itemize}
\item[ii.] $H(\tilde\phi)=H(\phi).$
\end{itemize}
\end{D}

Note that, for $g=2$ and $r=s=1$, the morphism $(0,0,1)$ is Gauss-reduced, but not Special. While $(1,0,2)$ is Quasi-special but not Special. In addition, for $g=r=2$, $s=1$,  $\phi=(I_2|{}^2_3)
$ is Quasi-special but not Gauss-reduced.

We want to show that if a point of large rank is in the kernel of a morphism then it is  in the kernel of a Quasi-special morphism.

\begin{lem}
\label{ind}
Let $V$ be an algebraic subvariety of $E^g$. Let $p=(p_1,\dots,p_s)$ be a point in $E^s $ of rank $s$. There exists $\varepsilon_0(p)>0$  depending on $p$ such that  for all $\varepsilon\le \varepsilon_0(p)$, for any subset $V^e$ of $V$  and positive integer $r$ it holds
$$S_r(V^e \times p,\mathcal{O}_\varepsilon)\subset
\bigcup_{\substack{\tilde\phi:E^{g+s} \to E^r \\  {\rm{Quasi-special}}}} (V^e \times p)\cap (B_{\tilde\phi}+\mathcal{O}_{\varepsilon} ).$$

\end{lem}
\begin{proof}
The proof is the analog of \cite{io} Lemma 6.2, where we shall read $V^e$ for $C$.

\end{proof}

\section{The proof of Theorem \ref{main}: The four main steps}
\label{passi}

In the following, we present the four main steps for the  proof of Theorem \ref{main}.

\begin{itemize}
\item[(0)] We prove Theorem \ref{equii} which claims  that Theorem \ref{main} i. and ii. are equivalent. 
\end{itemize}

We then shall prove Theorem \ref{main} ii.

\begin{itemize}
\item[(1)] In Proposition \ref{speciali} we get rid of $\Gamma$ by
  considering instead of $V$ the weak-transverse variety $V\times
  \gamma$. The key point is that    for $V\times \gamma$ we
  consider 
  \begin{equation*}
   \bigcup_{\substack{\tilde\phi:E^{g+s} \to E^{d+1}\\ \rm{Special} }}(V_{K}
 \times \gamma) \cap (B_{\tilde{\phi}}
  +\mathcal{O}_{\delta}) 
  \end{equation*}
  where  the union ranges only over {\bf{Special}} morphisms (and
  not over all  Gauss-reduced morphisms). 

\item[(2)] In Proposition \ref{centro} we show that the above union  is contained in the union of {\bf{finitely many}} sets of the kind
$$(V_{K}
 \times \gamma) \cap (B_{\tilde{\phi}}
  +\mathcal{O}_{\delta'/H(\tilde{\phi})^{\expalt}}). $$  Important is that the
  radius of the  neighbourhood of these finitely many sets is
  {\bf{inversally proportional to the height of the morphism}} (and it
  is not a constant $\delta$ like in the union in step (1)).
  
  \item[(3)] In Proposition \ref{finito} we show that  if the stabilizer of $V$ is finite, then there exists
    $\varepsilon>0$ such that, for all 
    Special morphisms $\tilde\phi$ of rank at least $d+1$, the  set $$(V_{K}
 \times \gamma) \cap (B_{\tilde{\phi}}
  +\mathcal{O}_{\delta/H(\tilde{\phi})})$$ is non-Zariski dense in $V\times \gamma$.

\end{itemize}

The statements (0), (1) and (2) are an immediate generalization of \cite{io} Theorem 1.3,
Proposition 10.2  and Proposition A respectively.  Part (3) is the most delicate and
it is presented in section \ref{due}, below.  It is the counterpart to  \cite{io} Proposition B. In order to  gain advantage from Theorem \ref{sin2}, we need to require that the stabilizer of the variety is finite. In view of Lemma \ref{stabi} this assumption will not be restrictive.  

\fbox{Part (0)} 
Theorem \ref{equii} is an immediate consequence of 
\begin{thm}
\label{equi34}
Let $V$ be an irreducible algebraic subvariety of $E^g$.  Then, for $\varepsilon\ge 0$ and $r$ a positive integer:

\begin{enumerate}
\item[i.]
 The map $x \to (x,\gamma)$ defines an injection
$$S_r(V,\Gamma_{\varepsilon})\hookrightarrow  S_r(V\times \gamma,\mathcal{O}_{\varepsilon}).$$ 
{\rm Recall that $\gamma$ is a maximal free set  of the division group $\Gamma_0$ of $\Gamma$}.
\end{enumerate}
Let $p\in E^s$ be a point of rank $s$ and $K\ge 0$. Then, there exists   $\varepsilon_0(p)>0$  such that:
\begin{enumerate}
\item[ii.]
 For    $ \varepsilon \le \varepsilon_0(p)$, the map $(x,p) \to x$ defines an injection
$$S_r(V_\kzero \times p,\mathcal{O}_{\varepsilon})\hookrightarrow
S_r(V_\kzero,(\Gamma^g_p)_{\varepsilon K'}),$$ where  
$K'=(g+s)\max \left(1,\frac{g(\kzero+\varepsilon)}{c(p)} \right)$  and $c(p)$ is a positive constant depending on $p$.\\
{\rm Recall that $\Gamma_p$ is the division group of the coordinates of $p$.}
  \end{enumerate}
\end{thm}
\begin{proof}
The proof is the analog of the proof of \cite{io} Theorem 9.1, where we
shall read $V$ for $C$,  $K$ for $K_3$, $\varepsilon_0(p)$ for $\varepsilon_p$ and $K'$ for $K_4$. Note that the inequality $||x||\le K$ is insured by 
considering just points in $V_\kzero$ (unlike in
\cite{io} where $||x||\le K_3$  is due to the hypothesis $r\ge2$ and
$\varepsilon\le \varepsilon_3$). 
\end{proof}

\fbox{Part (1)} 
Given a subgroup $\Gamma$ and a real $\kzero$, \cite{io} Lemma 3.4 (with $\emor(E)=\ze$) proves that there exists a maximal free set  $\gamma_{1}, \dots ,\gamma_{s}$ of the division group $\Gamma_0$ such that
\begin{equation}
\label{basegamma}
\begin{split}
||\gamma_i||&\ge 3gK,\\
\Big|\Big|\sum_i b_i\gamma_i\Big|\Big|^2&\ge \frac{1}{9}\sum_i|b_i|^2||\gamma_i||^2.
\end{split}
\end{equation}
for  $b_1, \dots, b_s\in \ze$.
We define  $$\gamma=(\gamma_{1}, \dots ,\gamma_{s})$$ with $\gamma_i$ satisfying the above conditions.
\begin{propo}
\label{speciali}
Let $V$ be an irreducible algebraic subvariety of $E^g$.    For  $r$ a positive integer, $K\ge0$ and $ \varepsilon \le \frac{K}{g}$, the map $x \to (x,\gamma)$ defines  an injection 
$$\bigcup_{\substack{\phi:E^g\to E^r\\ {\rm{Gauss \,\,\,reduced}}} } V_\kzero \cap \left(B_\phi+\left(\Gamma^g_0\right)_{\varepsilon}\right)
\hookrightarrow 
\bigcup_{\substack{\tilde\phi=( \phi|\phi')\\ {\rm{Special}} }}
(V_\kzero\times \gamma) \cap (B_{\tilde{\phi}} +\mathcal{O}_{\varepsilon}).$$

\end{propo}
\begin{proof}

 The proof is the analog of \cite{io} Proposition 10.2, where one
 shall read $K$ for $K_1$, $V$ for $C$, $V_\kzero$ for
 $C(\overline{\qe}) $. Note that, here, the estimate $||x||\le K$ is ensured
 by the assumption that we consider points in $V_\kzero$ (unlike
 in \cite{io}, where it is due to the assumptions $r\ge 2$ and $\varepsilon\le \varepsilon_1$).
\end{proof}

\fbox{Part (2)}  

\begin{propo} 
\label{centro}
 Let $V$ be an irreducible algebraic subvariety of $E^g$.  Let $p=(p_1,\dots ,p_s)\in E^s $ be a point of rank $s$. 
Then, for  $r$ a positive integer,  $K\ge 0$ and $\varepsilon >0$,

$$ \bigcup_{\substack{\tilde\fiffo:E^{g+s} \to E^r \\ {\rm{Special}} }}(V_\kzero\times p) \cap \left(B_{\tilde{\fiffo}} +\mathcal{O}_{\varepsilon/ {M}^\expaltu}\right) \subset 
 \bigcup_{\substack{\tilde\pippo:E^{g+s} \to E^r \\ {\rm{Special}},\,\,H(\tilde\pippo)\le M }}
 (
 V_\kzero\times p) \cap \left(B_{{\tilde\pippo}}
  +\mathcal{O}_{(g+s+1)\varepsilon/H(\tilde\pippo)^\expaltu}\right) ,$$
  where  $M=\max\left(2,\lceil\frac{\kdue}{\varepsilon}\rceil^2\right)^{n}$ and   $n=r(g+s)-r^2+1$.

\end{propo}

\begin{proof}
The proof is the analog of the proof of \cite{io} Proposition A part ii., where one shall read $V_\kzero$ instead of $C(\overline{\qe})$, $p$ for $\gamma$, $K$ for $K_2$ and $M$ for $M'$. And where the estimate $||x||\le K$ is ensured by the assumption that we consider points in $V_{\kzero}$ (and not as in \cite{io}, where it is due to the hypothesis  $r\ge 2$ and $\varepsilon\le \varepsilon_2$).

Note that in the last row of the proof in \cite{io} we estimate $g-r+1+s+1$ with $g+s$, because $r\ge2$.  Here we instead estimate $g-r+1+s+1$ with $g+s+1$, because $r\ge1$. 

\end{proof}

\section{The proof of Theorem \ref{main}: Part (3)}

\label{due}
Recall that 
$\mu(\varieta)$ is the supremum
of the   reals $\epsilon(\varieta)$ such that $\varieta\cap
\mathcal{O}_{\epsilon(\varieta)}$ is non-Zariski dense in $V$.  The essential minimum of $V$  is the square of $\mu(\varieta)$. 
 Using Theorem \ref{sin2}, we produce a sharp lower bound for   the essential minimum of the image of a variety under a Gauss-reduced morphism.
Unlike for curves, the stabilizer of the variety will play quite an important role. 
In this section, we will often assume that $V$ has finite stabilizer. In Lemma \ref{stabi}, we will see that such an assumption is not restrictive for the proof of our main theorem.

\subsection{The  estimate for the essential minimum}

Consider  a Gauss-reduced morphism $\phi$
of codimension $r=d+1$
\begin{equation*}
\phi= \left(\begin{array}{c}
\varphi_1\\
\vdots\\
\varphi_r
\end{array}\right)= 
 \left(\begin{array}{cccccc}
a & \dots& 0& L_1\\
\vdots & \ddots& \vdots & \vdots\\
0 & \dots& a& L_r
\end{array}\right)
\end{equation*}
where $L_i\in \ze^{g-r}$. 
We denote by $\overline{x}=(x_{r+1},\dots ,x_g)$.

We define the isogenies:
  \begin{equation}
  \label{hf}
  \begin{split}
  F: &E^{g} \to E^g\\
  & (x_1, \dots , x_g)\to(x_1,\dots,x_r,ax_{r+1}, \dots , ax_g).\\[0.3 cm]
   L:&E^g\to E^g\\
  &(x_1,\dots , x_g)\to (x_1+L_1(\overline{x}),\dots,x_r+L_r(\overline{x}),x_{r+1}, \dots ,x_g).\\[0.3 cm]
    \Phi:&E^g\to E^g\\
  &(x_1,\dots , x_g)\to (\varphi_1(x),\dots,\varphi_r(x),x_{r+1}, \dots ,x_g).
    \end{split}
  \end{equation}

\begin{D}[Helping-Variety]

We define the variety $$W=LF^{-1}(\varieta).$$
Then $$\Phi(\varieta)=[  a]W.$$
\end{D}

We now estimate degrees.
\begin{propo}
\label{dgr}
There exist positive constants $\cgf$ and $\cgw$ depending on $g$ and $\deg E$ such that:
\begin{itemize}

\item[i.]
The degree of $\phi(\varieta)$ is bounded by $\cgf a^{2d}\deg \varieta$.\end{itemize}

Suppose further that $V$ has finite stabilizer. Then,
\begin{itemize}

\item[ii.]
 The degree of $W$ is bounded by $\cgw a^{2(g-r)}|\stabv| \deg \varieta$.

\end{itemize}
\end{propo}

\begin{proof}
For simplicity we indicate by $\ll$ an inequality up to a multiplicative constant depending on $g$ and $\deg E$.

Let $X$ be an irreducible algebraic subvariety of $E^g$.

 First we estimate the degree of the image of $X$ under an isogeny $\psi:E^g \to E^g$.
 According to the chosen polarization 
 $$\deg \psi(X)=\sum _IE_{i_1}\cdot \dots \cdot E_{i_d}\cdot \psi(X),$$
 where $I=(i_1,\dots, i_d)$ ranges over the possible combinations of $d$ elements in the set $\{1,\dots ,g\}$ and $E_{i_j}$ is the coordinate subgroup given by $x_{i_j}=0$. Then
 $$\deg \psi(X)\ll \max_I\big( E_{i_1}\cdot \dots \cdot E_{i_d}\cdot \psi(X)\big).$$
 Let us estimate the intersection numbers on the right. By definition
 $$E_{i_1}\cdot \dots \cdot E_{i_d}\cdot \psi(X)=B_{\psi_I}\cdot X$$
 where the rows of $\psi_I$ are the $i_1,\dots, i_d$ rows of $\psi$. Note that ${\rm{rk\,\, }}\psi_I = d$ and $H(\psi_I)\le H(\psi)$.
 Bezout's Theorem and Lemma \ref{degb} (applied with $\phi=\psi_I$ and $r=d$)  give
 $$B_{\psi_I}\cdot X\le \deg B_{\psi_I} \deg X\ll H(\psi_I)^{2d}\deg X\ll H(\psi)^{2d}\deg X.$$
  We conclude 
\begin{equation*}
\deg \psi(X)\ll H(\psi)^{2d}\deg X.
\end{equation*}

 For $\psi=\Phi$, we deduce
\begin{equation}
\label{df}\deg  \Phi(V)\ll H(\Phi)^{2d}\deg V= a^{2d} \deg V.\end{equation}

i. In the chosen polarization, forgetting coordinates makes degrees decrease.

Note that $\phi(V)=\pi\Phi(V)$, where $\pi$ is the projection on the first $r$ coordinates.
By (\ref{df}) we conclude that
$$\deg  \phi(V)\le \deg \Phi(V)\ll  a^{2d} \deg V.$$

 ii. In   \cite{Hin} Lemma 6 part i. Hindry proves:
 
 For any positive integer $b$, 
  $$ \deg [b] X = \frac{b^{2d}}{|{\rm{Stab}}X \cap E^g[b]|} \deg X, $$
 where $|\cdot|$ means the cardinality of a set and $E^g[b]$ is  the kernel of the multiplication  $[b]$.

Recall that $\Phi(V)=[a]W$.
We deduce that
$$\deg \Phi(V)=\deg [a]W =\frac{a^{2d}}{|{\rm{Stab}}W \cap E^g[a]|} \deg W.$$
Thus $$\deg W= \frac{|{\rm{Stab}}W \cap E^g[a]|}{a^{2d}} \deg \Phi (V).$$
By relation (\ref{df}) we deduce
\begin{equation}
\label{stb}
\deg W \ll |{\rm{Stab}}W \cap E^g[a]|\deg V.
\end{equation}
We now estimate the cardinality of the stabilizer of $W$. Since $W=LF^{-1}V$,  we get
$${\rm{Stab}}W= LF^{-1}\stabv.$$
More precisely, if $x \in \stab W$ then $x+W\subset W$. Recall that $L$ is an isomorphism. Applying $FL^{-1}$ on both sides, we obtain $FL^{-1}x+V\subset V$. Thus $FL^{-1} x \in \stab V$ and $x \in LF^{-1}\stab V$. On the other hand, suppose that $x\in LF^{-1}\stab V$. Then $FL^{-1}x+V\subset V$. Considering the preimage,  $x+\ker (FL^{-1})+W\subset W$. But, by definition, $W$ is $\ker (FL^{-1})$ invariant, so $x+W\subset W $ and $x \in \stab W$.

By assumption the stabilizer of $V$ is finite. In addition $L$ is an isomorphism. So
$$|{\rm{Stab}}W|=|\ker F| |\stabv|=a^{2(g-r)}|\stabv|.$$
In view of (\ref{stb}), we conclude that
$$\deg W \ll   |{\rm{Stab}}W|\deg V \ll a^{2(g-r)}|\stabv| \deg V.$$

\end{proof}

The following Proposition is a lower bound  for  the essential minimum of the image of a variety under Gauss-reduced morphisms.  It reveals the dependence on the height of the morphism.  While the first bound is an immediate application of Theorem \ref{sin2} and Proposition \ref{dgr}, the second estimate is subtle. 

\begin{propo}
\label{EM}
 Let $\phi$ be a Gauss-reduced morphism of rank $d+1$ with $a=H(\phi)$.  Then,
 for any point $y\in E^g $ and any $\eta>0$,   \begin{itemize}
 \item[i.]  
  \begin{equation*}
\mu(\phi(V+\punto))>\ecuno \frac{1}{a^{\coefa}},
\end{equation*}
where $\ecuno$ is an effective  positive constant depending on $V$, $E$, $g$ and $\eta$. 
\end{itemize}
Suppose further that $V$ has finite stabilizer. Let $\Phi$ be the isogeny defined in (\ref{hf}).  Then,
\begin{itemize}
\item [ii.] 
  \begin{equation*}
  \mu\left(\Phi(V+\punto)\right )>\ecdue a^{\frac{1}{g-d}-2(g-d-1)\eta},
  \end{equation*}
  where  $\ecdue$ is an effective  positive constant depending on $V$,  $E$, $g$ and $\eta$. 
 \end{itemize} 

\end{propo}

\begin{proof}
Let us recall  the Bogomolov type bound given in Theorem \ref{sin2}; for a transverse irreducible variety $X$ in $E^g$ and any $\eta>0$ 
\begin{equation}
\label{BB} \mu(X)>\epsilon(X,\eta)=\frac{c(E^g,\eta)}{\deg X^{\frac{1}{2{\rm{cod}}X}+\eta}}.\end{equation}
 
 \vspace{0.3 cm}
 i.  
Let $q=\phi(\punto)$. Then $\phi(V+\punto)= \phi(V)+q$.  Since $V$ is irreducible, transverse and defined over $\overline{\qe}$,  $\phi(V)+q$ is as well.

Observe that $\phi(V)\subset E^{d+1}$ has dimension at least $1$ (because $V$ is transverse) and at most $d$ (because dimension can just decrease under morphisms). Furthermore dimensions are preserved by translations.

The  bound (\ref{BB}) for $\phi(V)+q$ and $g=d+1$ gives
\begin{equation*}
\begin{split}
\mu(\phi(V+\punto))&= \mu(\phi(V)+q)\\&> \epsilon(\phi(V)+q,\eta)=\frac{c(E^{d+1},\eta)}{\left(\deg (\phi(V)+q)\right)^{\frac{1}{2\cod \phi(V)}+\eta}}\\
&\ge \frac{c(E^{d+1},\eta)}{\left(\deg (\phi(V)+q)\right)^{\frac{1}{2}+\eta}}.
\end{split}
\end{equation*}
Degrees are preserved by translations,
hence Proposition \ref{dgr} i. implies  $$\deg (\phi(V)+q)=\deg\phi(V)\le \cgf a^{2d}\deg V.$$
If follows
$$\epsilon(\phi(V)+q,\eta)\ge \frac{c(E^{d+1},\eta)}{(\cgf a^{2d}\deg V)^{\frac{1}{2}+\eta}}.$$
Define
$$ \ecuno=\frac{c(E^{d+1},\eta)}{(\cgf \deg V)^{\frac{1}{2}+\eta}}.$$
Then
$$\mu(\phi(V+\punto))> \frac{\ecuno}{a^{d+2d\eta}}.$$

 \vspace{0.3 cm}

ii.
Let $q\in E^g $ be a point such that $[a]q=\Phi(\punto)$.  Let $W_0$ be an irreducible component of $W=LF^{-1}(V)$. Then
 $$\Phi(V+\punto)=[a](W_0+q).$$
 Therefore 
 \begin{equation}
 \label{essmin}
 {\mu\left(\Phi(V+\punto )\right)}=a\mu(W_0+q).
 \end{equation}
 
We now estimate  $\mu(W_0+q)$  via the  bound (\ref{BB}).
The variety $W_0+q\subset E^g$ is irreducible by  definition. Since $V$ is transverse and defined over $\overline{\qe}$, $W_0+q$ is as well.  Furthermore, isogenies and translations preserve dimensions. Thus $\dim (W_0+q)=\dim V=d$.
Then,  
$$\mu(W_0+q)>  \epsilon(W_0+q,\eta)=\frac{c(E^g,\eta)}{\deg(W_0+q)^{\frac{1}{2(g-d)}+\eta}}.$$
Since $W_0$ is an irreducible component of $W$, $\deg W_0\le \deg W$. Furthermore,
 translations by a point preserve degrees. Thus Proposition \ref{dgr} ii. with $r=d+1$ gives $$\deg (W_0+q)\le \deg W\le \cgw a^{2(g-d-1)}|\stabv|\deg V.$$
 Therefor 
$$ \mu(W_0+q)> \frac{ c(E^g,\eta)}{(\cgw |\stabv|\deg V)^{\frac{1}{2(g-d)}+\eta}}\left(a^{2(g-d-1)}\right) ^{-\frac{1}{2(g-d)}-\eta}.$$
Define $$\ecdue= \frac{ c(E^g,\eta)}{(\cgw|\stabv|\deg V)^{\frac{1}{2(g-d)}+\eta}}.$$
So
$$ \mu(W_0+q)> \ecdue{a^{-1+\frac{1}{g-d}-2(g-d-1)\eta}}.$$
Replace in (\ref{essmin}), to obtain
$$\mu(\Phi(V+\punto))> \ecdue a^{\frac{1}{g-d}-2(g-d-1)\eta}.$$

\end{proof}

\subsection{The non-density of the intersections}
We come to the  main proposition of this section: each set in the
union is non-Zariski dense. The proof of i. case (1)  is delicate. In
general $\mu(\pi(V))\le \mu(V)$ for $\pi$ a projection on some factors. We shall rather find a kind of reverse inequality.  On a set of bounded height this will be possible.

\begin{propo}
\label{finito}
 Suppose that $V\subset E^g$ has finite stabilizer.
Then, for every $K\ge0$, there exists an effective $\ecinque>0$ such that:

\begin{enumerate}
\item
For $\varepsilon \le \ecinque$, for
 all Gauss-reduced morphisms $\phi:E^g \to E^{d+1}$ and for all $\punto
 \in E^{d+1}\times \{0\}^{g-d-1}$, the set 
$$\left(V_\kzero+\punto\right)\cap \left(B_\phi+\mathcal{O}_{\varepsilon/H(\phi) }\right)$$  
is non-Zariski dense in $V$.

\item

Let $s$ be a positive integer. For $\varepsilon \le \frac{\ecinque}{g+s}$, for all Special morphisms $\tilde{\phi}=( \phi|\phi'):E^{g+s} \to E^{d+1}$  and for all points $p\in E^s$, 
 the set 
$$(V_\kzero \times p) \cap \left(B_{\tilde{\phi}}
  +\mathcal{O}_{\varepsilon/H(\phi) }\right) $$ is non-Zariski dense in $V\times p$.

\end{enumerate}

\end{propo}
\begin{proof}

Define  
 \begin{equation*}
 \begin{split}
 \eta&=\frac{1}{\valeta},\\
 m&=\left(\frac{\kzero}{\ecdue
  }\right)^{\frac{g-d}{1-\coeta(g-d)\eta}},\\
  \ecinque&= \min \left(\frac{\kzero}{g}, \frac{ \ecuno }{gm^{\expm}}\right),\\
   \end{split}
  \end{equation*}
 where  $\ecuno$ and $\ecdue$ are as in Proposition \ref{EM}.

 {\bf Part i.}
 
Let $a=H(\phi)$. We distinguish two cases:
\begin{itemize}
\item[(1)] $a \ge m$,
\item[(2)]  $a \le m.$
\end{itemize}

Case (1) - \fbox{$a\ge m$}

Let $x +\punto  \in  (V_{\kzero}+\punto )\cap \left(B_\phi+\mathcal{O}_{\varepsilon/a }\right)$, where $$y=(y_1,\dots, y_{d+1},0,\dots,0)\in E^{d+1}\times\{0\}^{g-d-1}.$$ Then
$$\phi(x+y)=\phi(\xi)$$  for  $||\xi||\le \varepsilon/a $.

Let  $\Phi=\phi\times id_{E^{g-d-1}}$ as in (\ref{hf}). Then
\begin{equation*}
\begin{split}
\Phi(x+\punto )&=(\phi(x+y),x_{d+2}, \dots ,x_g)\\
&=(\phi(\xi), x_{d+2},\dots , x_g).
\end{split}
\end{equation*}
 Therefore
$$ ||\Phi(x+\punto )||=||(\phi(\xi), x_{d+2},\dots , x_g)||\le \max\left(||\phi(\xi)||,||x||\right).$$
Since $||\xi||\le \frac{\varepsilon}{a}$ and
$\varepsilon \le \frac{K}{g}$, then $$||\phi(\xi)||\le  g{\varepsilon}\le K.$$
Also $||x||\le K$, because $x\in V_\kzero$. Thus
 $$||\Phi(x+\punto )||\le \kzero.$$
 We work under the hypothesis  $a \ge m\ge \left(\frac{\kzero}{\ecdue }\right)^{\frac{g-d}{1-\coeta(g-d)\eta}}$, then 
 \begin{equation*}
  \kzero \le\ecdue a^{\frac{1}{g-d}-\coeta\eta}.
   \end{equation*}
 In Proposition \ref{EM} ii. we have proven
 $$\ecdue
 a^{\frac{1}{g-d}-\coeta\eta}<\mu(\Phi(V+\punto)).$$  So  
 $$||\Phi(x+\punto )||\le K< \mu(\Phi(V+\punto)).$$ 
 We deduce that $\Phi(x+\punto )$ belongs to the  non-Zariski dense set $$Z_1=\Phi(V+\punto)\cap {\mathcal{O}}_{K}.$$ 
 The restriction morphism $\Phi_{|V+\punto}:V+\punto \to \Phi(V+\punto)$ is finite, because $\Phi$ is an isogeny. Then 
 $x+y$ belongs to the non-Zariski dense set $\Phi_{|V+\punto}^{-1}(Z_1)$.

 We can conclude that, for every $\phi$ Gauss-reduced of rank  $d+1$ with $H(\phi)\ge m$, the set
$$(V_\kzero+\punto )\cap \left(B_\phi+\mathcal{O}_{\varepsilon/H(\phi) }\right)$$ is non-Zariski dense.

\vspace{0.3cm}

Case (2) -  \fbox{$a\le  m$}

 Let $x+\punto  \in (V_\kzero+\punto )\cap
(B_\phi+\mathcal{O}_{\varepsilon/a })$, where $y \in E^{d+1}\times\{0\}^{g-d-1}$.   Then  $$\phi(x+\punto )=\phi(\xi)$$
for $||\xi ||\le \varepsilon/a $. However we have chosen  $\varepsilon \le\ecuno /gm^{\expm}$. Hence 
$$||\phi(x+\punto )||=||\phi(\xi)||  \le {g\varepsilon} \le\frac{ \ecuno }{m^\expm }.$$
We are working under the hypothesis 
$a\le m$.
Moreover $\eta =\frac{1}{2d}$. 
Then
$$a^{\coefa}\le m^\expm.$$
Thus
$$||\phi(x+\punto )||\le\frac{ \ecuno }{m^\expm }\le \frac{ \ecuno }{a^{\coefa}}.$$
In Proposition \ref{EM} i.  we have proven
$$\frac{ \ecuno }{a^{\coefa}}<\mu(\phi(V+\punto)).$$
We deduce that  $\phi(x+\punto )$ belongs to the non-Zariski dense set $$Z_2=\phi(V+\punto)\cap {\mathcal{O}}_{\ecuno /m^{\expm} }.$$
Since $V$ is transverse, the dimension of $\phi(V+\punto)$ is at least $1$. Consider  the restriction  morphism $\phi_{|V+\punto}:V+\punto \to \phi(V+\punto)$. Then $x+y$ belongs to the non-Zariski dense set $\phi^{-1}_{|V+\punto}(Z_2)$.

We conclude that,  for all $\phi$ Gauss-reduced of rank  $d+1$ with $H(\phi) \le m$, the set 
$$(V_\kzero+\punto )\cap \left(B_\phi+\mathcal{O}_{\varepsilon/H(\phi) }\right)$$ is non-Zariski dense.\\

Cases (1) and (2) prove part i. \\

{\bf Part ii.} 
 We are going to  show that, for every $\tilde\phi=( \phi|\phi')$ Special of rank ${d+1}$ (note that $\phi$ is Gauss-reduced of
 rank ${d+1}$), there exists
  $\punto\in E^{d+1}\times \{0\}^{g-d-1}$   such
 that the map $(x,p)\to x+y$ defines an injection
 \begin{equation}
 \label{undici}
 (V_\kzero\times p) \cap \left(B_{\tilde{\phi}}
 +\mathcal{O}_{\varepsilon/H(\phi) }\right) \hookrightarrow 
 (V_\kzero+\punto )\cap
 \left(B_\phi+\mathcal{O}_{(g+s)\varepsilon/H(\phi) }\right). 
 \end{equation}
 We then apply  part i. of this proposition to $\phi$ and $y$;   since $(g+s)\varepsilon \le \ecinque$, then
$$(V_\kzero+\punto )\cap
\left(B_\phi+\mathcal{O}_{(g+s)\varepsilon/H(\phi) }\right) $$ is non-Zariski dense in $V$. So
 for $\varepsilon\le  \frac{\ecinque}{g+s}$,  the set 
$$(V_\kzero\times p) \cap \left(B_{\tilde{\phi}} +\mathcal{O}_{\varepsilon/H(\phi) }\right) $$
is non-Zariski dense in $V$.
 
Let us prove the inclusion (\ref{undici}).
Let $\tilde\phi=( \phi|\phi')$ be Special of rank $d+1$. By definition of Special $\phi=(aI_{d+1}|L)$ is Gauss-reduced of rank ${d+1}$.

Let ${\punto'}\in E^{d+1}$ be a point such that $$[a]{\punto'}=\phi'(p).$$ 
Define $$\punto=(\punto',0,\cdots ,0)\in E^{d+1}\times \{0\}^{g-d-1}$$
Then
 $$\phi(\punto)=[a]\punto'=\phi'(p).$$ 
 
Let $$(x,p) \in (V_\kzero\times p) \cap
\left(B_{\tilde{\phi}} +\mathcal{O}_{\varepsilon/a}\right) .$$ Then, there exists  $\xi \in \mathcal{O}_{\varepsilon/a }$ such that  $$\tilde\phi((x,p)+\xi)=0.$$ 
Equivalently $$\phi(x)+\phi'(p)+\tilde\phi(\xi)=0$$
and
$$\phi(x+y)+\tilde\phi(\xi)=0.$$ 

Let $\xi''\in E^{d+1}$ be a point such that $$[a]\xi''=\tilde\phi(\xi).$$
We define $\xi'=(\xi'', \{0\}^{g-d-1})$, then
$$  \phi(\xi')=[a]\xi''=\tilde\phi(\xi),$$
and $$  \phi(x+\punto+\xi')=0.$$
Since $\tilde\phi$ is Special $H(\tilde\phi)=a$. Further $||\xi||\le \frac{\varepsilon}{a }$. We deduce
$$||\xi'||=||\xi''||=\frac{||\tilde\phi(\xi)||}{a}\le
\frac{(g+s)\varepsilon}{a }.$$

 In conclusion  $$ \phi(x+\punto +\xi')=0$$ 
 with $||\xi'||\le  \frac{(g+s)\varepsilon}{a }$.
 Equivalently $$(x+\punto ) \in (V_\kzero+\punto )\cap \left(B_{ \phi}+\mathcal{O}_{(g+s)\varepsilon/H(\phi) }\right) ,$$
 where ${\punto}\in E^{d+1}\times \{0\}^{g-d-1}$ and $\phi$ is Gauss-reduced of rank ${d+1}$.
 
 This proves relation (\ref{undici}) and concludes the proof.

  \end{proof}

\section{The Proof of Theorem \ref{main}: Conclusion}

\subsection*{Reducing to a variety with finite stabilizer}

In the following lemma, we will show  that to prove Theorem \ref{main} it is sufficient to prove it for varieties with finite stabilizer. This innocent remark will allow us to use all results of section \ref{due}.
\begin{lem}
\label{stabi}
They hold:

\begin{enumerate}

\item
Let $X=X_1\times E^{d_2}$ be a subvariety of $E^{g}$ of dimension d. Then, for $r\ge d_2$,  $$S_{r}(X,F)\hookrightarrow S_{r-d_2}(X_1,F')\times E^{d_2}$$
where $F'$ is the projection of $F$ on $E^{g-d_2}$.

\item
Let $V$ be a  (weak)-transverse subvariety of   $E^{g}$.
Suppose that $\dim \stabv =d_2\ge1$.
Then, there exists an isogeny $j$ of $E^g$ such that 
$$j(V) =V_1\times E^{d_2}$$ with $V_1$ (weak)-transverse in $E^{g-d_2}$ and $\stabv_1$ a finite group.

\item Theorem \ref{main}   holds if and only if it holds for  varieties with finite stabilizer.
\end{enumerate}

\end{lem}

\begin{proof}
i.
Let $(x_1,x_2)\in S_{r}(X,F)$ with $x_1\in X_1$ and $x_2\in
E^{d_2}$. Then, there exist $\phi:E^{g} \to E^{r}$  of rank $r$ and $(f_1,f_2)\in F$ such that
\begin{equation}
\label{gig}
\phi((x_1,x_2)-(f_1,f_2))=0.
\end{equation}
Decompose $\phi=(\alpha|\beta)$ with $\alpha:E^{g-d_2}\to E^{r}$ and $\beta:E^{d_2} \to
E^{r}$. 
Note that  ${\rm{rk\,\, }} \beta=r_2\le d_2$ because of the number of columns. 
Then, the Gauss algorithm ensures the existence of an invertible matrix $\Delta\in {\rm{GL}}_r(\ze)$ such that 
\begin{equation*}
\Delta \phi=\left(
\begin{array}{cc}
\phi_1&0\\
\star &\phi_2
\end{array}
\right),
\end{equation*}
where $\phi_1:E^{g-d_2}\to E^{r-r_2}$ 
and $\phi_2:E^{d_2} \to E^{r_2}$ of  rank $r_2$.

Since 
$r={\rm{rk\,\, }}\phi= {\rm{rk\,\, }} \phi_1+{\rm{rk\,\, }} \phi_2$,  we deduce 
${\rm{rk\,\, }} \,\phi_1=r-r_2\ge r-d_2$. Furthermore, relation (\ref{gig}) implies
 $$\phi_1(x_1-f_1)=0.$$ Thus $x_1 \in S_{r-d_2}(X_1,F')$.
 
\vspace{0.3cm}

ii.  
Let $\stabvz$ be the zero component of $\stabv$.
Consider the projection
$$\pi_S:E^g \to E^g/\stabvz.$$ Define $V'_1=\pi_S(V)$.
Then $$\dim V'_1= \dim (V+\stabvz)-\dim \stabvz=d-d_2<g-d_2.$$
Since $V$ is (weak)-transverse and $\dim V'_1<g-d_2$, then $V'_1$ is  (weak)-transverse in $E^g/\stabvz$.
Let $ (\stabvz)^{\perp}$ be the orthogonal complement of $\stabvz $ in $E^g$ and let $j_0:E^g/\stabvz \to (\stabvz)^{\perp}$ be an isogeny.  Define  
 the isogeny \begin{equation*}
 \begin{split}j':& E^g \to \left(E^g/\stabvz\right) \times \stabvz\\& x \to (\pi_S(x), x-j_0(\pi_S(x)) .
 \end{split}
 \end{equation*}
  Then $$j'(V)\subset V'_1\times \stabvz.$$ Since these varieties have the same dimension and are irreducible
$$j'(V)= V'_1\times \stabvz.$$
Let $i_0:   E^g/\stabvz\to E^{g-d_2}$  and $i_1:  \stabvz \to E^{d_2}$ be isogenies. Define $i=i_0\times i_1$, $j=i \circ j'$ and $V_1=i(V'_1)$. Then $$j(V)=V_1\times E^{d_2},$$ with $V_1$ (weak)-transverse in $E^{g-d_2}$.
Finally $${\rm{Stab}}V_1=i \circ \pi_S (\stabv)$$  is finite.

iii.  Suppose that $V$ is (weak)-transverse in $E^g$ and that $\dim \stabv =d_2>0$, then, by part ii., we can fix an isogeny $j$ such that $j(V)=V_1\times E^{d_2}$ with $\stabv_1$ a finite group and $V_1$ (weak)-transverse in $E^{g-d_2}$ of dimension $d_1=d-d_2$. Furthermore, by part i. with $X=j(V)$, $X_1=V_1$, $r=d+1$ and $F=\Gamma_\varepsilon$, we know that 
$$S_{d+1} (V,\Gamma_\varepsilon)\hookrightarrow S_{d_1+1}(V_1,\Gamma'_\varepsilon)\times E^{d_2}.$$
 So, if  $S_{d_1+1}(V_1,\Gamma'_\varepsilon)$ is non-Zariski dense  in $V_1$ also $S_{d+1} (V,\Gamma_\varepsilon)$ is non-Zariski dense in $V$.
\end{proof}
We can now conclude the proof of our main theorem. Let us recall that in view of Theorem \ref{equii} it is sufficient to prove part ii.
\begin{proof}[{\bf Proof of Theorem \ref{main} ii.}]
In view of Lemma \ref{stabi} iii. we can assume that $\stabv$ is finite.
Recall that $r=d+1$, the rank of $\Gamma_0$ is $s$ and  $n=(d+1)(g+s)-(d+1)^2 +1$. Let $\gamma=(\gamma_1,\dots, \gamma_s)$ be a point of rank $s$, such that $\gamma_i$ is a maximal free set of $\Gamma_0$ satisfying conditions (\ref{basegamma}).

\vspace{0.2cm}
Choose
\begin{enumerate}
 
 \item 
$ \delta_1 =\frac{1}{(g+s+1)} \min (\frac{\ecinque}{g+s},K)$ \hfill
  where  $\ecinque$ is as in Proposition \ref{finito},\,\,\,\,
  \\
  \item $\delta= {\delta_1}{{M}^{-1-\frac{1}{2n}}}$  \hfill{where
 $M=\max\left(2,\lceil \frac{\kzero +||\gamma||}{\delta_1}\rceil^2\right)^{n}.$}

 \end{enumerate}
 
\vspace{0.2cm}

Since $\Gamma_\delta\subset (\Gamma_0^g)_\delta$, then
$$S_{d+1}(V_\kzero,\Gamma_{\delta})\subset S_{d+1}(V_\kzero,(\Gamma_0^g)_{\delta}).$$
 Lemma \ref{tor1}, with  $\varepsilon=\delta$ and $r=d+1$, shows that 
$$S_{d+1}(V_\kzero,(\Gamma_0^g)_{\delta}) = \bigcup_{\substack{\phi: E^g \to E^{d+1} \\ \rm{Gauss-reduced}}}  V_\kzero\cap (B_\phi+(\Gamma_0^g)_{\delta}).$$

\vspace{0.2cm}

Note that   $\delta<\delta_1\le \frac{\kzero}{g}$. Then, Proposition \ref{speciali} with  $\varepsilon=\delta$ implies  
$$\bigcup_{\substack{\phi:E^g\to E^{d+1}\\ {\rm{Gauss-reduced}}} } V_\kzero\cap (B_\phi+(\Gamma_0^g)_{\delta}) \hookrightarrow  \bigcup_
{\substack{\tilde{\phi}=(\phi|\phi')\\ {\rm{Special}}} } 
(V_\kzero
 \times \gamma) \cap (B_{\tilde{\phi}}
  +\mathcal{O}_{\delta}).$$
  
  \vspace{0.2cm}
  
Note that $\delta_1>0$ and $\delta= {\delta_1}{{M}^{-(\expalt)}}$. Then,  Proposition \ref{centro},  with   $\varepsilon=\delta_1$, $r={d+1}$ and $p=\gamma$ shows that 
$$ \bigcup_
{\substack{\tilde{\phi}:E^{g+s} \to E^{d+1}\\ {\rm{Special}}} } 
(V_\kzero
 \times \gamma) \cap (B_{\tilde{\phi}}
  +\mathcal{O}_{\delta})$$ is a subset of 
  \begin{equation*}
  Z=\bigcup_
  {\substack{\tilde{\phi}:E^{g+s} \to E^{d+1}\\ {\rm{Special}},\,\,H(\tilde\phi)\le M} } 
  (V_\kzero
 \times \gamma) \cap \left(B_{\tilde{\phi}}
  +\mathcal{O}_{(g+s+1)\delta_1/H(\tilde\phi)^\expalt}\right).
  \end{equation*}
 Observe that   $Z$ is the union of finitely many sets, because  $H(\tilde\phi)$ is bounded by $M$.
  
  \vspace{0.2cm}
  
 We have chosen $\delta_1 \le {\ecinque}/{(g+s+1)(g+s)}$, moreover $\stabv$ is finite. Then, Proposition \ref{finito}   ii., with $\varepsilon=(g+s+1)\delta_1\le \frac{\ecinque}{g+s}$ and $p=\gamma$, implies that  for all $\tilde\phi=( \phi|\phi')$ Special of rank ${d+1}$, the set $$(V_\kzero
 \times \gamma) \cap \left(B_{\tilde{\phi}}
  +\mathcal{O}_{(g+s+1)\delta_1/H(\phi) }\right)$$ is non-Zariski dense in $V\times \gamma$.  Note that $H(\phi)\le H(\tilde\phi)$, thus also the  sets $$(V_\kzero
 \times \gamma) \cap \left(B_{\tilde{\phi}}
  +\mathcal{O}_{(g+s+1)\delta_1/H(\tilde\phi)^\expalt}\right)$$  are  non-Zariski dense. So $Z$  is non-Zariski dense, because it is the   union of finitely many non-Zariski dense sets. 
  We conclude that  $S_{d+1}(V_\kzero,\Gamma_\delta)$ is included in the  non-Zariski dense set $Z$.

\end{proof}

\begin{remark}
In \cite{io} we defined a different helping-curve $W'=A_0^{-1}W$ with $W$ the helping-variety used here and $A_0=(I_2|a_0I_{g-2})$. This more complicated $W'$ is needed because in \cite{io} we produced  a worse bound for the degree of  $W'$. Consequently, we proved a `weak' proposition \ref{finito}: we needed to assume that the neighbourhoods have  radius $\varepsilon/a_0a$. To compensate this loss, we needed the `strong' proposition \ref{centro}, where the radius is $\varepsilon/a_0a$. This was sufficient to  prove  our main theorem for curves. Such a  trick is not sufficient to prove an  optimal result for varieties.

In the present work, using the stabilizer, we produce a `good' bound for the degree of $W$, and we can prove the `strong' proposition \ref{finito} for neighbourhoods  of radius $\varepsilon/a$. Then, to prove our main theorem in general, it is sufficient to use a `weak' proposition \ref{centro}, where the radius of the neighbourhoods is $\varepsilon/a$. 

 If we try to combine both `strong' statements, namely Proposition \ref{centro} (with $\varepsilon/a_0a$) and the `good' bound for the degree of  $W$, we do not get any relevant improvement. Indeed in the proof of proposition \ref{finito} part i. the inequality $||x||\le \kzero$ remains unchanged. The advantage would only be in respect of  $\varepsilon$ in the statement of proposition \ref{finito}, where we  could choose $\varepsilon\le \epsilon_1 m$.

\end{remark}

\section{A special case of Conjecture \ref{hb}}
\label{exg}

The natural rising question is to investigate the height property for the codimension of the algebraic subgroups at least $d+1$. We expect that Conjecture \ref{hb} holds.
The known  results regarding this conjecture are based on a Vojta inequality, unless $\Gamma$ is trivial.
Following R\'emond's work, we prove here a new case of conjecture \ref{hb}.
In this section $E$ is a general elliptic curve (never mind wether C.M. or not). In view of R\'emond \cite{RHG2} Proposition 5.1, we give the following:

\begin{D}
We say that a subset $V^e$ of $V $ satisfies a Vojta inequality if there exist real constants $c_1,c_2,c_3>0$ such that for $x_1,\dots,x_{d+1}\in V^e$ with $||x_i||\ge c_3$ and for $\phi$ Gauss-reduced of rank $r\le g$,  there exists $s_1,\dots,s_{d+1}\in \mathbb{N}^*$ with $s_i\ge c_2 s_{i+1} $ such that 
$$\sum_{i=2}^{d+1} ||s_i \phi(x_i)-s_1\phi(x_1)||^2 \ge \frac{H(\phi)^2}{c_1}\sum_{i=1}^{d+1} s_i^2||x_i||^2.$$
\end{D}
Note that a Gauss-reduced morphism is a normalized projector in the sense of \cite{RHG2}. Then,
 this definition tells us that if \cite{RHG2} Proposition 5.1 holds for points in $V^e$ then $V^e$ satisfies a Vojta inequality. 
\begin{thm} [R\'emond, \cite{RHG2} Theorem 1.2]
\label{Gal1}
  If $V^e\subset V $ satisfies a Vojta inequality,
 then there exists  $\varepsilon>0$   such that
  $S_{d+1}(V^e,\Gamma_\varepsilon)$ has bounded height.
  
\end{thm}

R\'emond also gives a definition of a candidate $V^e$ which satisfies a Vojta inequality and potentially is a non-empty open in $V$. In a recent article he shows:
\begin{thm} [R\'emond \cite{Gprep}]
\label{den}

Assume that $V\subset E^g$ satisfies condition (\ref{stella}).  Then there exists a non-empty open subset $V^u$ of $V$ such that $V^u $ satisfies a Vojta inequality.
\end{thm}
These two theorems imply:
\begin{thm}
\label{g1}
Conjecture \ref{hb} ii. holds for  $V$  satisfying  
condition (\ref{stella}).
\end{thm}

Here, we extend his theorem to the associated weak-transverse case.
\begin{thm}
\label{wtalt}
 
 Conjecture \ref{hb} i. holds for $V\times p$, where   $V$  satisfies
condition (\ref{stella}) and $p$  is a point in $E^s $  not lying in any proper algebraic
subgroup of $E^s$. 

\end{thm}

For $V$ transverse and $p\in E^s$ a point of rank $s$, we can not embed the set
$S_{\indice}(V\times p,\mathcal{O}_\varepsilon)$   in a set of the type
$S_{\indice}(V,\Gamma_{\varepsilon'})$, unless we know a priori that the first
  set has bounded height. So, Theorem \ref{Gal1} is not enough to deduce  a statement for $V\times p$.

However, we can embed 
$S_{\indice}(V\times p,\mathcal{O}_\varepsilon)$ in the union of two sets
$S_{\indice}(V,\Gamma_{\varepsilon'})\cup (V \cap G_{p,\varepsilon,r})$, where the set $G_{p, \varepsilon,r}$ is defined in the proof of Theorem \ref{sopra} below. The same method can be
used to show that, for $V^e$ satisfying a Vojta inequality, $V^e\cap G_{p,\varepsilon,r}$ has bounded height, exactly as we do for curves in \cite{io} Theorem 1.2.

Let us write the details.

 \begin{D}
 Let $r,s $ positive integers and $\varepsilon>0$ a real.
 Let $p$ be a point in $E^s$.
    We define  $G_p^{\varepsilon,r}$ as the set of points  $\g \in E^{r}$ for which there exist a matrix $A\in M_{r, s} (\emor(E))$, an element $a\in \emor(E)$ with  $0<|a|\le H(A)$, points $\xi \in E^s$ and $\upxi \in E^r$ of norm at most $\varepsilon$ such that 
    $$[a]\g={A}({p}+\xi) +[a]\upxi.$$
We identify $G_p^{\varepsilon,r}$ with the subset $G_p^{\varepsilon,r}\times\{0\}^{g-r}$ of $E^g$.  
\end{D}

\begin{lem}
\label{spezzo}
Let $V$ be a subvariety of $E^g$.
Let $V^e$ be a subset of $V $ and let $p\in E^s $ be a point. Then, for every $\varepsilon\ge 0$, the projection  on the first $g$ coordinates 
\begin{equation*}
\begin{split}
E^{g}\times E^s &\to E^g\\
(x,y) &\to x
\end{split}
\end{equation*}
 defines an injection
\begin{equation*}
 S_r(V^e\times {p},\mathcal{O}_{\varepsilon/2gs})\hookrightarrow V^e\cap \bigcup_{\substack{\phi:E^g\to E^r \\{\rm{Gauss-reduced}}}} \left(B_\phi+(\Gamma_p^g)_\varepsilon\right)\cup\left(B_\phi+G_p^{\varepsilon,r}\right).
 \end{equation*}
\end{lem}

\begin{proof}
The proof is the analog of the proof of \cite{io} Lemma 7.2, where we shall replace $C(\overline{\qe})$ by $V^e$, the codimension $2$ by $r$ (as well as $E^2$ and  $g-2$ by $E^r$ and $g-r$),  the set $G^\varepsilon_p$ by $G^{\varepsilon,r}_p$. Also we shall use Lemma \ref{ind} stated in this article, instead of \cite{io} Lemma 6.2 to which we refer there.
\end{proof}
  
  \begin{lem} [Counterpart to \cite{RHG2} Lemma 6.1]
  \label{3.2}
   For $\phi:E^g \to E^r$ Gauss-reduced of rank $r$, we have the following inclusion of sets 
\begin{equation*}
(B_\phi+G_p^{\varepsilon,r})  {\subset}
\{ P+ \g \,\,:\,\,P\in B_\phi,\,\,\g\in G_p^{\varepsilon,r}  \,\,\,\,{\rm{and}}\,\,\,\max(||\g||,||P||)\le 2g||P+\g||\}.
\end{equation*}
 
 \end{lem}
 \begin{proof}
The proof is the analog of \cite{io} Lemma 7.3, where one replaces $G^\varepsilon_p$ by $G_p^{\varepsilon,r}$  and $2$ by $r$.

\end{proof}

Note that, \cite{RHG2} Lemma 6.2 part (1) is a statement on the morphisms, therefore it holds with no need of any remarks.

\begin{lem}[Counterpart to \cite{RHG2} Lemma 6.2 part (2)]
\label{3.3} Let $c_1$ be a given constant. Let $p\in E^s$ be a point of rank $s$.
There exists $\euno>0$ such that  if $\varepsilon \le \euno$ then any sequence of elements in $G_p^{\varepsilon,r}$ admits  a sub-sequence in which every two elements $\g$, $\g'$ satisfy 
$$\left|\left|\frac{\g}{||\g||}-\frac{\g'}{||\g'||}\right|\right|\le \frac{1}{16 g c_1}.$$
\end{lem}
\begin{proof}
The proof is the analog of \cite{io} Lemma 7.4 where $A, A' \in M_{r,s}(\emor(E))$ and $A=\left(\begin{array}{c}A_1\\\vdots\\A_r\end{array}\right),$ with $A_i\in M_{1,s}(\emor(E))$. \end{proof}

We are ready to conclude.

\begin{thm}
\label{sopra}
Let $p\in E^s$ be a point of rank $s$.
Suppose that   $V^e\subset V$  satisfies a Vojta inequality.  Then, there exists $\varepsilon>0$ such that $$S_{d+1}(V^e\times p, \mathcal{O}_\varepsilon)$$ has bounded height.
\end{thm}
\begin{proof}
Define $$\Gamma_{\varepsilon,r} = \bigcup_{\substack{\phi:E^g\to E^r \\{\rm{Gauss-reduced}}}} \left(B_\phi+(\Gamma_p^g)_\varepsilon\right)$$ and 
$$G_{p,\varepsilon,r}=\bigcup_{\substack{\phi:E^g\to E^r \\{\rm{Gauss-reduced}}}} \left(B_\phi+G_p^{\varepsilon,r}\right).$$
In view of  Lemma \ref{spezzo},  
$S_{d+1}(V^e\times p, \mathcal{O}_\varepsilon) \hookrightarrow \left(V^e\cap \Gamma_{\varepsilon,d+1}\right)\cup   \left(V^e\cap G_{p,\varepsilon,d+1}\right)$.

 Theorem \ref{Gal1}  shows that there exists $\varepsilon_1>0$ such that for $\varepsilon \le \varepsilon_1$, $V^e\cap \Gamma_{\varepsilon,d+1}=S_d(V^e, \Gamma_\varepsilon)$ has bounded height.
 
It remains to show, that there exists $\varepsilon_2>0$ such that  for $\varepsilon\le \varepsilon_2$, the set 
$V^e\cap G_{p,\varepsilon,d+1}$ has bounded height.
The proof follows, step by step, the proof of R\'emond 
 \cite{RHG2}  Theorem 1.2 page 341-343 where one shall read $G_{p,\varepsilon,r}$ for $\Gamma_{\varepsilon,r}$, 
$\g$  for $\gamma$, 
 $V^e$ for 
 $X(\overline{\qe})\setminus Z^{(r)}_X$.  Note that  he writes $|\cdot|$ for the height norm, here  we  write $||\cdot||$. For the morphisms he uses a norm denoted by $||\cdot||$, here we denote the norm of a morphism by $H(\cdot)$.
 \cite{RHG2} Lemmas 6.1 and 6.2 are replaced by  our  Lemmas \ref{3.2} and \ref{3.3}.
 The Vojta Inequality  \cite{RHG2} Proposition 5.1 holds for  the set $V^e$ by assumption. 
  \end{proof}

\begin{proof}[{\bf Proof of Theorem \ref{wtalt}}]
Thanks to Theorem \ref{den}  there exists a non-empty open subset $V^u$ of $V$ such that $V^u$ satisfies a Vojta inequality.  Theorem \ref{sopra} applied with $V^e=V^u$ implies that  there exists $\varepsilon>0$ such that $S_{d+1}(V^u\times p, \mathcal{O}_\varepsilon)$ has bounded height.
 \end{proof}

In conclusion Conjecture \ref{hb} i. and ii. are not equivalent, but the
same method can be applied to prove both cases.

\vskip1cm



\end{document}